\documentclass{amsart}
\usepackage{amsmath, amsfonts, amsthm}
\usepackage[T1]{fontenc}
\usepackage[french,english]{babel}
\usepackage{lmodern}
\usepackage{hyperref}
\usepackage{cleveref}
\numberwithin{equation}{section}
\newtheorem{lem}[equation]{Lemma}

\newtheorem{thm}[equation]{Theorem}
\newtheorem{cor}[equation]{Corollary}
\theoremstyle{definition}
\newtheorem{defn}[equation]{Definition}

\newcommand{\pref}[1]{\Cref{#1}}
\newcommand{\abs}[1]{\left\lvert#1\right\rvert}     
\newcommand{\comps}[1]{\left\langle#1\right\rangle} 
\newcommand{\cplx}{\mathbb{C}}                      

\newcommand{\lie}[1]{\mathfrak{#1}}                 
\newcommand{\inprod}[2]                             
    {\left\langle#1,#2\right\rangle}
\newcommand{\norm}[1]{\left\|#1\right\|}
\newcommand{\ints}{\mathbb{Z}}                      
\newcommand{\real}{\mathbb{R}}                      
\newcommand{\unv}[1]                                
    {\mathcal{U}\!\left(\lie{#1}\right)}            
\newcommand{\repf}{\mathcal{R}}                     
\DeclareMathOperator{\ad}{ad}

\DeclareMathOperator{\ann}{ann}

\DeclareMathOperator{\Irr}{Irr}

\DeclareMathOperator{\supp}{supp}
\begin{document}
\title[On the kernel of the maximal flat Radon transform]
    {On the kernel of the maximal flat Radon transform on symmetric spaces
    of compact type}
\author{Eric L. Grinberg}
\author{Steven Glenn Jackson}
    \address{Department of Mathematics \\
             University of Massachusetts \\
             100 Morrissey Boulevard \\
             Boston, MA 02125 \\
             USA}
    \email{\href{mailto:jackson@math.umb.edu}{jackson@math.umb.edu}}
    \urladdr{\href{http://www.math.umb.edu/~jackson/}{http://www.math.umb.edu/~jackson/}}
\begin{abstract}
    Let $M$ be a Riemannian globally symmetric space of compact type, $M'$ its
    set of maximal flat totally geodesic tori, and $\ad(M)$ its adjoint space.
    We show that the kernel of the maximal flat Radon transform $\tau:L^2(M)
    \rightarrow L^2(M')$ is precisely the orthogonal complement of the image
    of the pullback map $L^2(\ad(M))\rightarrow L^2(M)$.  In particular, we
    show that the maximal flat Radon transform is injective if and only if $M$
    coincides with its adjoint space.
\end{abstract}
\maketitle


\section{Introduction}

As part of his 1911 doctoral dissertation (published as \cite{funk13}), Paul
Funk considered the problem of recovering a function $f$ on the sphere $S^2$
from its integrals over great circles.  He showed that precisely the even part
$f^+(x)=\frac{1}{2} \left( f(x) + f(-x) \right)$ of $f$ can be recovered from
these integrals, while the odd part is annihilated by them.  Thus was born the
Funk transform---integration over closed geodesics in a Riemannian manifold.

Funk's original motivation was a problem in differential geometry suggested to
him by his advisor David Hilbert: the study of surfaces all of whose geodesics
are closed.  The sphere is the obvious example of such a surface, and Darboux
around 1884 had given a necessary condition for other surfaces of revolution
to possess this property.\footnote{It is difficult now to assign an exact date
to Darboux's condition.  In \cite{tan92}, Tannery writes:
\foreignlanguage{french}{\emph{``J'y ai \'et\'e conduit en \'etudiant la note XV
de la M\'ecanique de Despayrons, o\`u M.~Darboux a donn\'e une r\`egle
pour trouver les surfaces de r\'evolution admettant des lignes
g\'eod\'esiques ferm\'ees.''}}  However, the authors of the present paper
were unable to obtain this work; indeed, we could locate no reference to
it but this and an entry in the June 24, 1905 issue of the
\emph{Publisher's Circular} dating it to 1884--85.  The earliest
publication of Darboux's result still readily obtainable seems to be
\cite{darb94}.}
In 1892, Tannery had constructed an explicit, though singular, example
(\cite{tan92}), and in 1901 Otto Zoll, another of Hilbert's students, had
constructed smooth examples (published as \cite{zoll03}).  Funk's dissertation
sought to identify smooth deformations of the standard metric on $S^2$, all of
whose geodesics remain closed at each stage of the deformation.

Funk's result on integration over great circles implied that no such
deformation could exist if the deformed metrics were required to remain even
(nowadays we say that the standard metric on $\mathbb{RP}^2$ admits no such
deformations).  Conversely, given any odd function $h$, Funk tried to
construct by power series a conformal deformation of the standard metric with
initial derivative $h$, but was unable to prove convergence of his series.
Convergence of such deformations was finally proved by Guilleman in 1976
(\cite{guillemin76}).  Thus, for $S^2$ and $\mathbb{RP}^2$ at least, the
kernel of the Funk transform controls the ``Zoll rigidity'' of the
manifold.\footnote{To be precise, let us say that $g$ is a \emph{$P_l$-metric}
if all its geodesics are periodic with least period $l$; that $\{g_t\}$ is a
\emph{$P_l$-deformation} of $g_0$ if $g_t$ is a $P_l$-metric for all $t$; that
the deformation $\{g_t\}$ is \emph{trivial} if $(M,g_t)$ is isometric to
$(M,g_0)$ for all $t$; and that $g$ is \emph{$P_l$-rigid} (or \emph{Zoll
rigid}) if every $P_l$-deformation of $g$ is trivial.  Then Funk's results
showed that $\mathbb{RP}^2$ is Zoll rigid, while Guillemin's convergence
argument showed that every odd function gives rise to a 
$P_l$-deformation of $S^2$.} These results were extended to $S^n$ and
$\mathbb{RP}^n$, for arbitrary $n$, by Michel (\cite{mich73}) and Weinstein
(see \cite[p.~120]{besse78}).

Around this time, similar integral transforms arose in several related
contexts.  A smooth family $\{g_t\}$ of Riemannian metrics on a compact
manifold $M$ is said to be \emph{isospectral} if the Laplace spectrum of
$(M,g_t)$ is independent of $t$.  Such families are plentiful: given any local
one-parameter group of diffeomorphisms $\phi_t:M\rightarrow M$ and any metric
$g$ on $M$, one obtains an isospectral family by setting $g_t$ equal to the
pullback of $g$ along $\phi_t$.  Examples of this type are trivial in the
sense that $(M,g_t)$ is isometric to $(M,g)$ for every $t$.  One says that
$(M,g)$ is \emph{spectrally rigid} if every isospectral deformation is trivial
in this sense, and \emph{spectrally rigid to first order} if every isospectral
deformation agrees with a trivial one to first order in the deformation
parameter $t$.

In the case of Riemannian globally symmetric spaces of compact type, results
of Guillemin in \cite{gui79} implied that spectral rigidity is related to a
certain natural ``thickening'' of the Funk transform.  To be specific, let us
refer to any integral transform obtained by integrating over a fixed class of
totally geodesic submanifolds of fixed dimension as a ``Radon transform.''
(Thus the Funk transform is the Radon transform associated with compact
submanifolds of dimension one.)  For the case where $M$ is a Riemannian
globally symmetric space of compact type, Guillemin's results implied that the
initial derivative of any isospectral deformation must lie in the kernel of
the Radon transform associated with maximal totally geodesic flat tori.  One
says that $M$ is \emph{Guillemin rigid} if the Lie derivatives of the metric
(i.e.\ the initial derivatives of trivial deformations) fill the entire kernel
of the maximal flat Radon transform for symmetric two-forms.\footnote{See
\cite{gasqui04} for the definition of the Radon transform on symmetric
$p$-forms.}  Then Guillemin rigidity is a sufficient condition for first-order
spectral rigidity.  Since the kernel of the Funk transform is contained in the
kernel of the maximal flat Radon transform, an observation of Michel
(\cite[Proposition~2.2.4]{mich73}) implies that Guillemin rigidity is also a
sufficient condition for first-order Zoll rigidity of $P_l$-metrics.

Each Riemannian globally symmetric space $M$ of compact type admits a unique
\emph{adjoint form} (or \emph{universal covered space}): another symmetric
space admitting $M$ as a Riemannian cover, and not itself properly covering
any other symmetric space.  (For example, the adjoint form of the sphere $S^n$
is the real projective space $\mathbb{RP}^n$.) In \cite{gri92}, Grinberg
conjectured that the maximal flat Radon transform for functions is injective
if and only if $M$ coincides with its adjoint form.

In a monumental study (\cite{gasqui04}), Gasqui and Goldschmidt determined
precisely which Grassmannians (over $\mathbb{R}$, $\mathbb{C}$, or 
$\mathbb{H}$) are Guillemin rigid.  Aware of Grinberg's conjecture, they
observed that their criterion is equivalent to the requirement that $M$
coincide with its adjoint form, and went on to prove that, if Grinberg's
conjecture is true, then the only candidates for Guillemin rigidity are
the irreducible adjoint spaces.

In the present paper, we prove a refinement of Grinberg's conjecture: the
kernel of the maximal flat Radon transform consists precisely of those
functions orthogonal to all functions pulled back from the adjoint form
(\pref{thm:main}).  Applied to the sphere $S^2$, this recovers Funk's result
of 1911: the kernel of the Funk transform consists of those functions
orthogonal to all functions pulled back from $\mathbb{RP}^2$ (i.e.\ the
functions orthogonal to all even functions, namely the odd functions).  Our
proof is based on elementary representation theory of the isometry group of
$M$, and treats all Riemannian globally symmetric spaces of compact type
uniformly; there is no case-by-case analysis.

The organization of the paper is as follows: in \pref{sec:prelim} we fix
notation and gather basic tools; in \pref{sec:reynolds} we show that
the maximal flat Radon transform can be viewed as a Reynolds operator
for a certain group of isometries of $M$; in \pref{sec:support} we
relate this Reynolds operator to highest weight theory; in \pref{sec:adjoint}
we relate the resulting highest weight theory to the extraction of
the adjoint form; and in \pref{sec:main} we use our accumulated results
to prove Grinberg's conjecture.


\section{Preliminaries}
\label{sec:prelim}

In this section, we fix notation and gather some useful results from the
    literature.

\subsection{Symmetric spaces and isometry groups}

Let $M$ be a Riemannian globally symmetric space of compact type, and let $G$ be
    the identity component of the isometry group of $M$.
Then $G$ is a compact semisimple Lie group (\cite[Lemma~IV.3.2, Definitions~V.1,
    and Definition~V.4]{helgason78}).
Denote its Lie algebra by $\lie{g}_0$, and put $\lie{g}=\cplx\otimes_{\real}%
    \lie{g}_0$.

Fix a point $p\in M$, let $\sigma$ be the geodesic symmetry of $M$ at $p$,
    let $K$ be the isotropy group of $p$ in $G$, and let $K_0$ be the identity
    component of $K$.
Define $\theta:G\rightarrow G$ by $\theta(g)=\sigma g\sigma^{-1}$; then
    $\theta$ is an involution of $G$.
Letting $K_{\theta}$ be the set of fixed points of $\theta$, we have
    $K_0\subseteq K\subseteq K_{\theta}$ (\cite[Theorem~IV.3.3]{helgason78}).
These three groups share a common Lie algebra $\lie{k}_0$ (namely the set of
    fixed points of $d\theta$), whose complexification we denote by $\lie{k}$.

Since $G$ is transitive on $M$ (\cite[Theorem~IV.3.3]{helgason78}), we have
    $$M\simeq G/K.$$

\subsection{The adjoint form}

Let $Z$ be the center of $G$, and define $$K_Z=\{g\in G\,|\,\theta(g)^{-1}g\in%
    Z\}.$$
Evidently this is a subgroup of $G$ containing $K$.
Since $G$ is semisimple, $Z$ is finite, so $K$ has finite index in $K_Z$.

Let $\ad(G)=G/Z$ denote the adjoint group of $G$.
Since $\theta$ stabilizes $Z$, we have an induced involution on $\ad(G)$, whose
    fixed point set we denote by $\ad(G)^{\theta}$.
One shows easily that $K_Z$ is the pre-image of $\ad(G)^{\theta}$ under
    the natural projection $G\rightarrow\ad(G)$; hence the space
    $$\ad(M)=G/K_Z\simeq \ad(G)/\ad(G)^{\theta}$$ is again a Riemannian
    globally symmetric space of compact type; we refer to it as the
    \emph{adjoint form} of $M$.

Since $K$ has finite index in $K_Z$, the natural map $\pi:M\rightarrow\ad(M)$
    is a covering map.
In fact, $\ad(M)$ is minimal in the isogeny class of $M$, by which we mean that
    it is covered by every Riemannian globally symmetric space whose
    universal cover is isometric to that of $M$ (\cite[Corollary~VII.9.3]
    {helgason78}).

\subsection{Representative functions}

Let $H$ be any closed subgroup of $G$.
Since $G$ and $H$ are both compact, they are both unimodular and the
    homogeneous space $G/H$ admits a unique normalized left Haar measure
    $\mu_{G/H}$ (\cite[Theorem~29E, Lemma~30A, and Theorem~33D]{loomis53}).

Define an action of $G$ on $L^2(G/H)$ by the formula $(gf)(x)=f(g^{-1}x)$.
Since the inner product $(f_1,f_2)=\int f_1\overline{f_2}\,d\mu_{G/H}$ is
    invariant under this action, $L^2(G/H)$ is a unitary (hence continuous)
    representation of $G$.

Let $C(G/H)$ denote the algebra of continuous functions on $G/H$.
Let $\repf(G/H)$ denote the set of $G$-finite vectors in $C(G/H)$.
The elements of $\repf(G/H)$ are called \emph{representative functions}
    on $G/H$.
Since $G$ is compact, we have the inclusions $$\repf(G/H)\subseteq%
    C(G/H)\subseteq L^2(G/H).$$
By the Peter-Weyl Theorem (\cite[Theorem~III.5.7]{bd85}), the algebra
    $\repf(G/H)$ is dense in $L^2(G/H)$. 

In the case where $H$ is the trivial subgroup, we can say more.
Define an action of $G\times G$ on $L^2(G)$ by the formula $((g_1,g_2)f)(x)%
    =f(g_1^{-1}xg_2)$.
Then (again since $G$ is unimodular) this action is also unitary.
Let $e$ denote the identity element of $G$, and let $LG$ and $RG$
    denote the groups $G\times\{e\}$ and $\{e\}\times G$, respectively.
Then $f\in L^2(G)$ is $LG$-finite if and only if it is $RG$-finite
    (\cite[Proposition~III.1.2]{bd85}).
Moreover, \cite[Proposition~III.1.5]{bd85} gives the $G\times G$-module
    decomposition $$\repf(G) \simeq\bigoplus_{V\in\Irr(G)} V\otimes V^*$$
    where $\Irr(G)$ denotes the set of isomorphism classes of irreducible
    representations of $G$.
(Since $G$ is compact, it follows from the Peter-Weyl Theorem that all of its
    irreducible representations are finite-dimensional.)

Returning to the case of arbitrary closed $H$, we note that the natural
    projection $G\rightarrow G/H$ induces an injection $\repf(G/H)%
    \rightarrow\repf(G)$ intertwining the (left) action of $G$.
Identifying $\repf(G/H)$ with its image under this injection, we shall
    regard $\repf(G/H)$ as a subalgebra of $\repf(G)$.
From \cite[Example~III.6.3]{bd85} we have 
    $$\repf(G/H) \simeq\bigoplus_{V\in\Irr(G)}V\otimes\left(V^*\right)^H$$
    where $(V^*)^H$ denotes the space of $H$-invariant vectors in $V^*$.
More generally, whenever $H_1\subseteq H_2$ are both closed subgroups of
    $G$, we regard $\repf(G/H_2)$ as a subalgebra of $\repf(G/H_1)$ (and hence
    also a subalgebra of $\repf(G)$).

One says that $H$ is a \emph{spherical subgroup} of $G$ if
    $\dim\left((V^*)^H\right) \leq1$ for all $V\in\Irr(G)$.
By passing to the complexified Lie algebra $\lie{k}$ and using the proof of
    \cite[Theorem~12.3.12]{gw09}, one sees that $K_0$ is spherical.
Hence $K$ and $K_Z$ are also spherical.

\subsection{The torus transform}

Let $T$ be a maximal flat totally geodesic torus of $M$ containing $p$.
Let $\lie{a}_0$ be the tangent space to $T$ at $p$, and let $\lie{p}_0%
    \subset\lie{g}_0$ be the $-1$-eigenspace of $d\theta$.
By \cite[Proposition~V.6.1]{helgason78}, the differential of the map
    $g\mapsto gp$ identifies $\lie{a}_0$ with a maximal abelian subspace of
    $\lie{p}_0$.

By \cite[Theorem V.6.2]{helgason78}, $G$ is transitive on the set $M'$ of all
    maximal flat totally geodesic tori in $M$.
Let $L$ be the stabilizer of $T$ in $G$.
Then $M'$ is in bijective correspondence with the homogeneous space $G/L$, with
    which we shall now identify it.

Let $A=\exp(\lie{a}_0)$ be the Lie subgroup of $G$ with Lie algebra $\lie{a}_0$,
    and let $\overline{A}$ be the closure of $A$.
Then $\overline{A}$ is a compact abelian subgroup of $G$ whose elements
    satisfy the equation $\theta(a)=a^{-1}$.
It follows that the Lie algebra of $\overline{A}$ is an abelian subspace of
    $\lie{p}_0$ containing $\lie{a}_0$.
But $\lie{a}_0$ is maximal abelian in $\lie{p}_0$, so we have equality, and
    $A$ is in fact a closed subgroup of $G$ (hence a torus).

Since $T$ is totally geodesic, it follows from \cite[Theorem~IV.3.3]
    {helgason78} that the geodesics of $T$ passing through $p$ are precisely
    the orbits of $p$ under one-parameter subgroups of $A$.
But every point of $T$ is joined to $p$ by some such geodesic, so it follows
    that $A$ acts transitively on $T$.
Evidently the full stabilizer $L$ is also transitive on $T$, so we can write
    $$T\simeq L/(L\cap K)\simeq A/(A\cap K).$$

The normalized Haar measure $\mu_{L/(L\cap K)}$ is an $A$-invariant
    measure on $T$, so the uniqueness of such measures forces $\mu_{L/(L\cap K)}%
    =\mu_{A/(A\cap K)}$.
For brevity, let us denote this measure by $\mu_T$, and define a map
    $\tau_0:\repf(M)\rightarrow\repf(M')$ by the formula
    $$(\tau_0(f))(gT) = \int_{x\in T} f(gx)\,d\mu_T(x).$$
That $\tau_0(f)$ depends only on the torus $gT$ and not on $g$ itself follows
    from the $L$-invariance of $\mu_T$; that it depends continuously on $gT$
    follows from the compactness of $G$ and $T$; and that it is $G$-finite follows
    from the fact that $\tau_0$ intertwines the left action of $G$.

In the sequel we shall show that $\norm{\tau_0(f)}_2\leq\norm{f}_2$
    (\pref{cor:bounded}), from which it will follow that $\tau_0$ admits a unique
    bounded extension $\tau:L^2(M)\rightarrow L^2(M')$.
This extension is known in the literature as the \emph{maximal flat Radon
    transform} on $L^2(M)$; we shall call it the \emph{torus transform} for
    short.

\subsection{Weights and representations}

Let $\lie{h}_0$ be any maximal abelian subspace of $\lie{g}_0$ containing
    $\lie{a}_0$.
Then by \cite[Lemma~VI.3.2]{helgason78}, $\lie{h}_0$ is a Cartan subalgebra of
    $\lie{g}_0$, and it is stabilized by the action of $\theta$.
Denote by $\lie{t}_0$ be the set of $\theta$-fixed points in $\lie{h}_0$.
Since $\lie{a}_0$ is maximal among abelian subspaces of $\lie{p}_0$, we have
    $\lie{h}_0 = \lie{t}_0 \oplus \lie{a}_0$.
Let $H$ denote the maximal torus of $G$ with Lie algebra $\lie{h}_0$.

Next let $\lie{a}$, $\lie{t}$, and $\lie{h}$ denote the complexifications of
    $\lie{a}_0$, $\lie{t}_0$, and $\lie{h}_0$, respectively.
The complexification of $d\theta$ stabilizes $\lie{a}$, $\lie{t}$,
    and $\lie{h}$.
From now on we shall use the same symbol $\theta$ to denote the involution
    of $G$, its differential, the complexification of its differential,
    and the induced involutions on the duals of all the vector spaces on
    which these act.

Let $\Delta\subset\lie{h}^*$ be the root system of $\lie{g}$ with respect to
    $\lie{h}$, and let $\Sigma\subset\lie{a}^*$ be the restricted root system
    of the pair $(\lie{g},\lie{k})$ (i.e.\ the set of non-zero restrictions
    of members of $\Delta$ to $\lie{a}$).
Choose a positive system $\Delta^+$ for $\Delta$ such that the non-zero
    restrictions of its members form a positive system $\Sigma^+$ for
    $\Sigma$.
Let $\Pi\subseteq\Delta^+$ and $\Pi'\subseteq\Sigma^+$ be the resulting sets
    of simple roots.

Denote by $B$ the Killing form of $\lie{g}$.
Since $B$ has non-degenerate restriction to $\lie{h}$, it defines an isomorphism
    $\lie{h}\rightarrow\lie{h}^*$, by means of which we transfer $B$ to an inner
    product $(\cdot,\cdot)$ on $\lie{h}^*$.
Then $B$ is negative definite on $\lie{h}_0$, but $(\cdot,\cdot)$ is positive
    definite on $\real\Delta$, the real span of $\Delta$
    (\cite[Proposition II.6.6, Corollary~II.6.7, and
    Theorem~III.4.4]{helgason78}).
For any $\alpha,\beta\in\real\Delta$ with $\beta\neq0$, define $$\comps{\alpha,\beta}=%
    \frac{2(\alpha,\beta)}{(\beta,\beta)}.$$

Since the restriction of $B$ to $\lie{a}$ is also non-degenerate, we can use
    the same procedure to define an inner product on $\lie{a}^*$.
With this inner product, $\lie{a}^*$ is naturally isometric to $\ann(\lie{t})%
    \subseteq\lie{h}^*$, with which we shall now identify it.

Let $W$ denote the normalizer of $\lie{h}_0$ in $G$, modulo its centralizer.
Then $W$ acts faithfully on $\lie{h}_0$, hence also on $\lie{h}$ and its dual.
One finds that $W$ is generated by the root reflections $\{s_{\alpha}\,|\,%
    \alpha\in\Delta\}$; hence $W$ stabilizes $\real\Delta$.
We say that a weight $\lambda\in\real\Delta$ is \emph{dominant} if $(\lambda,%
    \alpha)\geq0$ for all $\alpha\in\Pi$.
Each element of $\real\Delta$ is $W$-conjugate to a unique dominant weight
    (\cite[Theorem~10.3 and Lemma~10.3B]{hum72}).  

Similarly, let $W_{\lie{a}}$ denote the normalizer of $\lie{a}_0$ in $K$,
    modulo its centralizer.
Then $W_{\lie{a}}$ acts faithfully on $\lie{a}_0$ and hence also on $\lie{a}$
    and its dual.
It is generated by the root reflections $\{s_{\alpha}\,|\,\alpha\in\Sigma\}$,
    so it stabilizes $\real\Sigma$.
A restricted weight $\nu\in\real\Sigma$ is \emph{dominant} if $(\nu,\alpha)%
    \geq0$ for all $\alpha\in\Pi'$.\
Each element of $\real\Sigma$ is $W_{\lie{a}}$-conjugate to a unique dominant
    restricted weight (\cite[Corollary~VII.2.13 and Theorem~VII.2.22]{helgason78}).

We say that a weight $\lambda\in\real\Delta$ is \emph{algebraically integral}
    if $\comps{\lambda,\alpha}\in\ints$ for all $\alpha\in\Pi$.
A weight $\omega\in\real\Delta$ is dominant and algebraically integral if and
    only if it is the highest weight of some finite-dimensional irreducible
    representation $V(\omega)$ of $\lie{g}$.
In this case every weight of $V(\omega)$ is algebraically integral
    (\cite[Theorem~5.5]{knapp02}).

We say that a weight $\lambda\in\real\Delta$ is \emph{analytically integral}
    if it is the complexified differential of a character of $H$.
In this case it is harmless to use the same symbol $\lambda$ to denote
    the character of which it is the differential, and to say that a subset
    $S\subseteq H$ is \emph{annihilated by $\lambda$} if $\lambda(s)=1$ for all
    $s\in S$.
Every analytically integral weight is algebraically integral.
A weight $\omega\in\real\Delta$ is dominant and analytically integral if and only
    if it is the highest weight of the differential of some irreducible
    representation of $G$, which we shall also denote by $V(\omega)$.
In this case every weight of $V(\omega)$ is analytically integral
    (\cite[Proposition~4.59, Lemma~5.106, and Theorem~5.110]{knapp02}).


\section{Properties of the torus transform}
\label{sec:reynolds}

In this section we prove that $\tau_0$ can be extended to a bounded operator
    on $L^2(M)$, and that it coincides with the restriction of a Reynolds
    operator on $\repf(G)$.

Recall that $\repf(M)$ can be identified with the algebra of right $K$-invariants
    in $\repf(G)$, and $\repf(M')$ can be identified with the algebra of right
    $L$-invariants.
We have

\begin{thm}
\label{thm:reynolds}
Let $R_A:\repf(G)\rightarrow\repf(G)$ be the operator orthogonally projecting
    $\repf(G)$ onto its space of right $A$-invariants, namely
    $$(R_A(f))(x) = \int_{a\in A}f(xa)\,d\mu_A(a)$$ where $\mu_A$ denotes
    normalized Haar measure on $A$.
Then for any $f\in\repf(M)$, we have $$\tau_0(f) = R_A(f).$$
\end{thm}

\begin{proof}
The torus $T$ is isomorphic to $A/(A\cap K)$.
The positive linear functional $I:C(A)\rightarrow\cplx$ defined by $$I(f) =%
    \int_{x(A\cap K)\in A/(A\cap K)}\left(\int_{y\in(A\cap K)}f(xy)\,%
    d\mu_{A\cap K}(y)\right)\,d\mu_{A/(A\cap K)}(x)$$ is $A$-invariant,
    preserves monotone limits, and satisfies $I(1)=1$; consequently $I$
    coincides with Haar integration over $A$ (\cite[Definition~12A and Theorem~29D]{loomis53}).
If $f$ is right $K$-invariant, then the inner integral is just $f(x)$
    (this is depends only on the coset $x(A\cap K)$) and we have
\begin{align*}
(\tau_0(f))(g) &= \int_{x(A\cap K)\in A/(A\cap K)} f(gx)\,d\mu_{A/(A\cap K)}(x) \\
&= \int_{a\in A} f(ga)\,d\mu_A(a) \\
&= (R_A(f))(g).
\end{align*}
\end{proof}

\begin{cor}
\label{cor:bounded}
For any $f\in\repf(M)$, we have $\norm{\tau_0(f)}_2\leq\norm{f}_2$.
\end{cor}

\begin{proof}
As above, the positive linear functional $J:C(G)\rightarrow\cplx$ defined by
    $$J(f) = \int_{xK\in G/K}\left(\int_{y\in K}f(xy)\,d\mu_K(y)\right)\,%
    d\mu_{G/K}(x)$$ coincides with Haar integration on $G$.
It follows immediately that the injection $\repf(M)\rightarrow\repf(G)$ preserves
    $L^2$-norms.
Similarly, the injection $\repf(M')\rightarrow\repf(G)$ preserves $L^2$-norms.
But the identification of $\repf(M)$ and $\repf(M')$ with their images in $L^2(G)$
    sends $\tau_0$ to the restriction of an orthogonal projection operator.
\end{proof}


\section{Supports of $\lie{k}$-invariants}
\label{sec:support}

Let $\omega$ be an analytically integral dominant weight.
Since $K_0$ is a spherical subgroup of $G$, the irreducible $G$-module
    $V(\omega)$ admits at most a one-dimensional space of $\lie{k}$-invariants.
Let $v^{\omega}$ be a non-zero $\lie{k}$-invariant in $V(\omega)$, if one
    exists, and zero otherwise.
Write $$v^{\omega} = \sum_{\lambda} v^{\omega}_{\lambda}$$ where
    $v^{\omega}_{\lambda}$ is a weight vector of weight $\lambda$.

\begin{defn}
The \emph{support} of $\omega$ is the set $\supp(\omega)=
    \{\lambda\,|\,v^{\omega}_{\lambda}\neq0\}$.
\end{defn}

In this section, we give a necessary and sufficient condition for a weight
    $\lambda$ of $V(\omega)$ to lie in $\supp(\omega)$ (\pref{thm:support}).
We begin with a lemma.

\begin{lem}
\label{lem:helglem}
Let $\alpha$ be a positive root.
If $\theta(\alpha)\neq\alpha$ then $-\theta(\alpha)$ is also a positive
    root.
On the other hand, if $\theta(\alpha)=\alpha$ then the root spaces
    $\lie{g}_{\alpha}$ and $\lie{g}_{-\alpha}$ are both contained in $\lie{k}$.
\end{lem}

\begin{proof}
This is \cite{helgason78} Lemma~VI.3.3.
\end{proof}

We shall also need a few facts concerning a certain partial order on $\lie{a}^*$ (which, recall, we have identified with $\ann(\lie{t}))\subseteq\lie{h}^*$).

\begin{defn}
For $\lambda, \mu \in \lie{a}^*\simeq\ann\lie{t}\subseteq\lie{h}^*$, define $\lambda\preceq\mu$ if and only if $\mu-\lambda$ lies in the non-negative span of $\Pi'$.
\end{defn}

\begin{lem}
\label{lem:preceq}
Suppose $\lambda\in\lie{a}^*$, and let $\mu$ be the unique dominant $W_{\lie{a}}$-conjugate of $\lambda$.
Then $\lambda\preceq\mu$.
\end{lem}
\begin{proof}
The proof is essentially the same as the argument given in \cite[Theorem~VII.2.22]{helgason78}.
Let $\nu$ be any maximal element in the $W_{\lie{a}}$-orbit of $\lambda$.
Then for any $\alpha\in\Pi'$ we have $\nu-s_{\alpha}(\nu) = \comps{\nu,\alpha}\alpha.$
Maximality of $\nu$ now forces $\comps{\nu,\alpha}\geq0$.
Since $\alpha$ was arbitrary in $\Pi'$ this shows that $\nu$ is dominant, and hence $\nu=\mu$.
But now the $W_{\lie{a}}$-orbit is a finite poset with \emph{unique} maximal element $\mu$; hence $\mu$ is also a largest element and the lemma is proved.
\end{proof}

We now need a pair of lattices in $\ann(\lie{t})\subseteq\lie{h}^*$.

\begin{defn}
Let $\Lambda$ denote the set $\{\rho-\theta(\rho)\}$ where $\rho$ runs over the root lattice, and let $\widehat{\Lambda}$ denote the set $\{\mu-\theta(\mu)\}$ where $\mu$ runs over the lattice of algebraically integral weights.
\end{defn}

\begin{lem}
\label{lem:lattice}
Both $\Lambda$ and $\widehat{\Lambda}$ are discrete additive subgroups of $\ann(\lie{t})$ that span $\ann(\lie{t})$.
\end{lem}
\begin{proof}
Both are obviously additive subgroups.
The span of $\Lambda$ (resp.\ $\widehat{\Lambda}$) is all of $\ann(\lie{t})$ because the span of the root lattice (resp.\ the weight lattice) is all of $\lie{h}^*$, and the map $\mu\mapsto(\mu-\theta(\mu))$ is a surjective linear map of $\lie{h}^*$ onto $\ann(\lie{t})$.
Finally, $\Lambda$ (resp.\ $\widehat{\Lambda}$) is discrete because it is a subgroup of the root lattice (resp.\ the weight lattice).
\end{proof}
\begin{lem}
\label{lem:stable}
Both $\Lambda$ and $\widehat{\Lambda}$ are stable under the action of $W_a$.
\end{lem}
\begin{proof}
The group $W_{\lie{a}}$ is a quotient of the centralizer of $\theta$ in $W$ (\cite[Theorem~VII.8.10]{helgason78}).
But this centralizer evidently stabilizes both $\Lambda$ and $\widehat{\Lambda}$.
\end{proof}

\begin{lem}
\label{lem:highest_in_support}
Suppose that $V(\omega)$ admits a non-zero $\lie{k}$-invariant.
Then $\omega\in\supp(\omega)$.
\end{lem}
\begin{proof}
Let $\lambda$ be any element of $\supp(\omega)$, maximal with respect to $\preceq$.
If $\lambda=\omega$ then we are finished; otherwise, since $\lambda$ is not the highest weight of $V(\omega)$, we can find some positive root $\alpha$ such that $X_{\alpha}v^{\omega}_{\lambda}\neq0$ (where $X_{\alpha}$ is some non-zero element of the root space $\lie{g}_{\alpha}$).
If $\alpha$ is fixed by $\theta$, then \pref{lem:helglem} gives $X_{\alpha}
    \in\lie{k}$, contradicting the $\lie{k}$-invariance of $v^{\omega}$.
Consequently $\theta(\alpha)\neq\alpha$.
But $v^{\omega}$ must be annihilated by $X_{\alpha} + \theta(X_{\alpha})
    \in\lie{k}$, so we have $$\theta(X_{\alpha})(v^{\omega}_{\lambda+
    \alpha-\theta(\alpha)}) = -X_{\alpha}v^{\omega}_{\lambda},$$ forcing
    $\lambda+\alpha-\theta(\alpha)\in\supp(\omega)$.
Again by \pref{lem:helglem}, $-\theta(\alpha)$ is also a positive root, so $\lambda\prec\lambda+\alpha-\theta(\alpha)$, contradicting the maximality of $\lambda$.
\end{proof}

\begin{lem}
Suppose that $V(\omega)$ admits a non-zero $\lie{k}$-invariant.
Then $\omega\in\widehat{\Lambda}$.
\end{lem}
\begin{proof}
It suffices to consider the case in which $G$ is simply connected.
Consider the set $$F_0=\{a\in A\,|\,a^2=e\}$$ where $e$ denotes the identity element of $G$.
This is a subset of $K^{\theta}$.
Indeed, let $S$ denote the set of all $a\in\lie{a}$ such that $\mu(a)\in\pi i\ints$ for every algebraically integral weight $\mu$; then $F_0=\exp(S)$, so $F_0$ is actually a subset of $K_0$.
Then, since $\omega\in\supp(\omega)$, it follows that $F_0$ acts trivially on $v^{\omega}_{\omega}$.
Thus, for any $a\in S$ (i.e.\ any $a$ such that $\mu(a)\in2\pi i\ints$ for every  $\mu\in\widehat{\Lambda}$) we must also have $\omega(a)\in2\pi i\ints$.

Next, choose a generating set $\{\mu_1,\dots,\mu_k\}$ for $\widehat{\Lambda}$ such that $\{\mu_1,\dots,\mu_k\}$ is a basis for $\ann(\lie{t})$.
(This is possible by \pref{lem:lattice}.)
Then $\omega\in\supp(\omega)$ forces $\omega\in\ann(\lie{t})$, so we can write $$\omega=c_1\mu_1+\dots+c_k\mu_k$$ for some scalars $c_1,\dots,c_k$.
Let $b_1,\dots,b_k$ be the basis for $\lie{a}$ dual to $\mu_1,\dots,\mu_k$, and put $a_i=2\pi i b_i$, so that $a_1,\dots,a_k\in S$.
Then for each $i$, we have $\omega(a_i) = 2\pi ic_i\in2\pi i\ints$, showing that each $c_i$ is an integer.
\end{proof}

\begin{lem}
\label{lem:lhat_to_l}
Suppose that $\lambda\in\widehat{\Lambda}$, and $w\in W_{\lie{a}}$.
Then $\lambda-w(\lambda)\in\Lambda.$
\end{lem}
\begin{proof}
Choose an algebraically integral weight $\mu$ with $\lambda = \mu-\theta(\mu)$.
Since $\Lambda\subseteq\widehat{\Lambda}$ and $W_{\lie{a}}$ is generated by restricted root reflections, it suffices to prove the result in the case where $w=s_{\alpha|_{\lie{a}}}$ for some root $\alpha$ with $\theta(\alpha)\neq\alpha$.

There are two cases to consider.
Suppose first that $\alpha-\theta(\alpha)$ is proportional to some root $\beta$.
Then
\begin{align*}
\lambda-w(\lambda) &= \comps{\lambda,\beta}\beta \\
&= 2\comps{\mu,\beta}\beta \\
&= \comps{\mu,\beta}(\beta-\theta(\beta))
\end{align*}
and, since $\mu$ is algebraically integral, this lies in $\Lambda$.

On the other hand, suppose that $\alpha-\theta(\alpha)$ is not proportional to any root.
Then certainly $\theta(\alpha)$ is not proportional to $\alpha$.
But since exactly one of $\{\alpha,\theta(\alpha)\}$ is positive, we see that neither $\alpha-\theta(\alpha)$ nor $\alpha+\theta(\alpha)$ is a root, so \cite[Lemma~9.4]{hum72} forces $(\alpha,\theta(\alpha))=0.$
Then
\begin{align*}
\lambda-w(\lambda) &= 2\frac{(\lambda, \alpha-\theta(\alpha))}{(\alpha-\theta(\alpha),\alpha-\theta(\alpha))}(\alpha-\theta(\alpha)) \\
&= \frac{(\lambda,\alpha-\theta(\alpha))}{(\alpha,\alpha)}(\alpha-\theta(\alpha)) \\
&= 2\frac{(\lambda,\alpha)}{(\alpha,\alpha)}(\alpha-\theta(\alpha))
\end{align*}
and, since $\lambda$ itself is algebraically integral, this again belongs to $\Lambda$.
\end{proof}

At last we have

\begin{thm}
\label{thm:support}
Suppose that $V(\omega)$ admits a non-zero $\lie{k}$-invariant, and let
    $\lambda$ be any weight of $V(\omega)$.
Then $\lambda\in\supp(\omega)$ if and only if $\theta(\lambda)=-\lambda$
    and $\omega-\lambda$ lies in the lattice $\Lambda$.
\end{thm}

\begin{proof}
First suppose that $\lambda\in\supp(\omega)$.
Since $\lie{t}\subseteq \lie{k}$ and $v^{\omega}$ is $\lie{k}$-invariant, we have for any $t\in\lie{t}$ $$0 = tv^{\omega} = \sum_{\mu}\mu(t)v^{\omega}_{\mu}$$ from which it follows that $\lambda(t)=0$.
Then, for any $x\in\lie{h}$, write $x = x_+ + x_-$ with $x^+\in\lie{t}$
    and $x_-\in\lie{a}$.
We have
\begin{align*}
(\theta(\lambda))(x) &= \lambda(\theta(x_+ + x_-)) \\
&= \lambda(x_+-x_-) \\
&= \lambda(x_-) \\
&= \lambda(-x_+-x_-) \\
&= (-\lambda)(x)
\end{align*}
so that $\theta(\lambda)=-\lambda.$
The proof that $\omega-\lambda$ belongs to the lattice $\Lambda$ is by downward induction with respect to the partial order $\preceq$.
In the base case, $\omega-\lambda=0$ certainly belongs to $\Lambda$.
Otherwise $\lambda$ is not a highest weight of $V(\omega)$, so by the argument given in the proof of \pref{lem:highest_in_support}, there is some positive root $\alpha$ with $\lambda+\alpha-\theta(\alpha)\in\supp(\omega)$, and the result follows immediately by induction.

On the other hand, suppose that $\theta(\lambda)=-\lambda$ and $\omega-\lambda$
    belongs to $\Lambda$.
We shall prove that $\lambda\in\supp(\omega)$, again by downward induction with respect to $\preceq$.
The base case is \pref{lem:highest_in_support}.
For the general case, recall that $W_{\lie{a}}$ denotes the Weyl group of the pair $(G,K)$; in particular, each of its elements is a coset of an element of $K$.
It follows that $W_{\lie{a}}$ stabilizes the set
    $\{\mu|_{\lie{a}}\}_{\mu\in\supp{\omega}}$, so
    we may assume without loss of generality that $\lambda|_{\lie{a}}$ is dominant.
But this implies that $\lambda$ itself is a dominant weight of $\lie{g}$.
To see this, begin by noting that $\lambda\in\ann(\lie{t})$ since
    $\theta(\lambda)=-\lambda$.
Then for any positive root $\alpha$ we can write $\alpha = \alpha_+ + \alpha_-$
    with $\alpha_+\in\ann(\lie{a})$ and $\alpha_-\in\ann(\lie{t})$, giving
    $$(\lambda,\alpha) = (\lambda, \alpha_-) = (\lambda|_{\lie{a}}, \alpha|_{\lie{a}}) \geq 0$$
    and proving that $\lambda$ is dominant.
Now write $$\omega-\lambda = \sum_{\alpha\in\Delta^+}
    k_{\alpha}(\alpha-\theta(\alpha))$$ with $k_{\alpha}\in\ints$.
Since $\omega$ is the \emph{highest} weight of $V(\omega)$ we may take all
    $k_{\alpha}$ non-negative.
Evidently we may also take $k_{\alpha}=0$ whenever $\alpha$ is $\theta$-
    fixed.
Now choose $\alpha\in\Delta^+$ with $k_{\alpha}>0$.
Let $\lambda'$ be the unique dominant $W_a$-conjugate of $\lambda+\alpha-\theta(\alpha)$.
By \pref{lem:preceq}, we have $\lambda\preceq\lambda'$.
Evidently $\theta(\lambda')=-\lambda'$.
Choose $w\in W_{\lie{a}}$ with $$w(\lambda+\alpha-\theta(\alpha)) = \lambda'.$$
We have
\begin{align*}
\omega-\lambda' &= \omega - w(\lambda-\omega+\omega+\alpha-\theta(\alpha))\\
&= \omega - w(\omega) + w(\omega-\lambda) - w(\alpha-\theta(\alpha))
\end{align*}
which, by Lemmas~\ref{lem:stable} and \ref{lem:lhat_to_l}, lies in $\Lambda$.
Then by induction, $\lambda'\in\supp(\omega)$, so also $\lambda+\alpha-\theta(\alpha)\in\supp(\omega).$

Since $\theta(\alpha)\neq\alpha$, it follows from the Cauchy-Schwartz inequality that $(\theta(\alpha),\alpha)<\abs{\alpha}^2$.
Hence
\begin{align*}
(\lambda + \alpha - \theta(\alpha), \alpha) &= (\lambda,\alpha) + \abs{\alpha}^2 - (\theta(\alpha),\alpha) \\
&> (\lambda,\alpha) \\
&\geq0.
\end{align*}
from which it follows that $$X_{-\alpha}v^{\omega}_{\lambda+\alpha-\theta(\alpha)} \neq 0.$$
Now since $v^{\omega}$ is $\lie{k}$-invariant, we must have $$X_{-\alpha}v^{\omega} = -\theta(X_{-\alpha})v^{\omega}.$$
Extracting the component of each side of weight $\lambda-\theta(\alpha)$ gives $$X_{-\alpha}v^{\omega}_{\lambda+\alpha-\theta(\alpha)} = -\theta(X_{-\alpha})v^{\omega}_{\lambda}$$ forcing $v^{\omega}_{\lambda}\neq0$ and hence $\lambda\in\supp(\omega)$.

\end{proof}


\section{Representative functions on the adjoint form}
\label{sec:adjoint}

In this section we prove \pref{cor:inclusion} characterizing the image
    of the injection $\repf(\ad(M))\rightarrow\repf(M)$.
We begin with a trivial modification of \cite{kr71}, Proposition~1:

\begin{lem}
Let $F$ be the group $\{a\in A|a^2\in Z\}$.
Then $K_Z=FK_0$.
\end{lem}

\begin{proof}
Suppose first that $a\in F$.
We have $\theta(a)^{-1}a=a^2\in Z$, so $F\subseteq K_Z$.
Now $K_0$ is also the connected component of $K_Z$, so $F$ normalizes it
    and hence $FK_0$ is a group with $FK_0 \subseteq K_Z$.
On the other hand, suppose $g\in K_Z$.
Using the $K_0AK_0$ decomposition for $G$ (\cite[Theorem~V.6.7]{helgason78}),
    we can write $g=k_1ak_2$ for $k_i\in K_0$ and $a\in A$.
Then $z=\theta(g)^{-1}g=k_2^{-1}ak_1^{-1}k_1ak_2=k_2^{-1}a^2k_2$
    lies in the center of $G$.
But then $a^2=k_2zk_2^{-1}=z$, showing that $a\in F$.
Thus $g=k_1ak_2$ lies in $FK_0$.
\end{proof}

\begin{lem}
The group $F$ coincides with the set $\{\exp(a)\,|\,a\in\lie{a}\text{ and }%
    (\alpha-\theta(\alpha))(a)\in2\pi i\ints,\,\forall\alpha\in\Pi\}$.
\end{lem}

\begin{proof}
Evidently $F$ coincides with $\{\exp(a)\,|\,a\in\lie{a}\text{ and }%
    \alpha(2a)\in2\pi i\ints,\,\forall \alpha\in\Pi\}$.
But $(\theta(\alpha))(a)=-\alpha(a)$.
\end{proof}

This leads immediately to

\begin{cor}
\label{cor:gens}
Let $x_1,\dots,x_{\dim(\lie{a})}$ be the basis for $\lie{a}$ dual to $\Pi'$
    and put $S=\{\exp(\pi ix_1),\dots,\exp(\pi i x_{\dim(\lie{a})})\}$.
Then $F$ is generated by $S$.
\end{cor}

\begin{lem}
\label{lem:weightann}
Let $\lambda$ be an analytically integral weight annihilating $\lie{t}$.
Then $\lambda$ annihilates $F$ if and only if $\lambda$ belongs to the lattice
    $\Lambda$.
\end{lem}

\begin{proof}
Suppose first that $\lambda$ annihilates $F$.
In the notation of \pref{cor:gens}, we must have $\lambda(x_i)\in2\ints$.
Let $\{\alpha_i\}$ be a set of positive roots restricting to $\Pi'$.
Put $\mu = \sum_i(\frac{\lambda(x_i)}{2} (\alpha_i-\theta(\alpha_i))$.
Then $\lambda$ agrees with $\mu$ on both $\lie{t}$ and $\lie{a}$, so in
    fact $\lambda=\mu$.

On the other hand, suppose that $\lambda$ belongs to $\Lambda$.
It follows immediately that $\lambda(x_i)\in2\ints$ for all $i$, and hence
    that the corresponding weight of $G$ annihilates $F$.
\end{proof}

\begin{thm}
Let $V(\omega)$ be an irreducible representation of $G$ containing a non-zero
    $\lie{k}$-invariant.
Then $V(\omega)$ contains a non-zero $K_Z$-invariant if and only if $\omega%
    \in\Lambda$.
\end{thm}

\begin{proof}
Since both $K_Z$ and $K_0$ are spherical subgroups of $G$, we need only
    determine when $v^{\omega}$ is also $K_Z$-invariant.
This happens if and only if $v^{\omega}$ is $F$-invariant, which happens
    if and only if all the weights in $\supp(\omega)$ annihilate $F$.
By \pref{lem:weightann}, this happens if and only if every member of
    $\supp(\omega)$ belongs to $\Lambda$, and by \pref{thm:support}
    this happens if and only if $\omega$ itself belongs to $\Lambda$.
\end{proof}

\begin{cor}
\label{cor:inclusion}
The image of the injection $\repf(\ad(M))\rightarrow\repf(M)$ induced by the
    covering map consists precisely of those $V(\omega)$ contained in $\repf(M)$
    such that $\omega^*$, the highest weight of the dual of $V(\omega)$, belongs
    to $\Lambda$.
\end{cor}

\begin{proof}
The ring $\repf(M)$ consists of those $V(\omega)$ such that $V(\omega)^*$ contains
    a non-zero $K$-invariant, and the ring $\repf(\ad M)$ consists of those
    $V(\omega)$ such that $V(\omega)^*$ contains a non-zero $K_Z$-invariant.
\end{proof}

\section{The main result}
\label{sec:main}

\begin{thm}
\label{thm:main}
The kernel of the torus transform $\tau$ is the orthogonal complement of the
    image of the embedding $L^2(\ad(M))\rightarrow L^2(M)$ induced by the
    covering map.
\end{thm}

\begin{proof}
Since $\repf(M)$ is dense in $L^2(M)$, the kernel of $\tau$ is the closure
    of the kernel of $\tau_o$.
Since $A$ is connected, \pref{thm:reynolds} implies that $\ker\tau_0$ consists of
    all those $V(\omega)$ such that $\supp(\omega^*)$ contains no weight
    restricting to zero on $\lie{a}$.
But by \pref{thm:support}, these are precisely the $V(\omega)$ such that
    $\omega^*\not\in\Lambda$, and by \pref{cor:inclusion} these are precisely
    the $V(\omega)$ not in the image of the injection $\repf(\ad(M))%
    \rightarrow\repf(M)$ induced by the covering map.
The result follows by taking closures.
\end{proof}


\section*{Acknowledgement}

The authors wish to thank Todor Milev and Alfred No\"el for helpful
    conversations.

\bibliographystyle{amsalpha}
\bibliography{torus}
\end{document}